\documentclass[12pt]{article}

\usepackage{amstext}
\usepackage{amssymb}
\usepackage{amsmath}
\usepackage{pb-diagram} \usepackage{amsthm}
\usepackage{rotating}
\usepackage{amsfonts}
\usepackage{graphicx}
\theoremstyle{plain}
 \newtheorem{thm}{Theorem}
 \newtheorem{lem}{Lemma}
 \newtheorem*{thm*}{Theorem}
 \newtheorem*{lem*}{Lemma}
 \newtheorem{cor}{Corollary}
 \newtheorem*{cor*}{Corollary}
 \newtheorem{prop}{Proposition}
 
\theoremstyle{definition}

  \newtheorem*{Remark*}{Remark}

\begin{document}

\title{$2$-stratifold groups have solvable Word Problem}

\author{J. C. G\'{o}mez-Larra\~{n}aga\thanks{Centro de
Investigaci\'{o}n en Matem\'{a}ticas, A.P. 402, Guanajuato 36000, Gto. M\'{e}xico. jcarlos@cimat.mx} \and F.
Gonz\'alez-Acu\~na\thanks{Instituto de Matem\'aticas, UNAM, 62210 Cuernavaca, Morelos,
M\'{e}xico and Centro de
Investigaci\'{o}n en Matem\'{a}ticas, A.P. 402, Guanajuato 36000,
Gto. M\'{e}xico. fico@math.unam.mx} \and Wolfgang
Heil\thanks{Department of Mathematics, Florida State University,
Tallahasee, FL 32306, USA. heil@math.fsu.edu}}
\date{}

\maketitle

\begin{center}
{\em To Professor Maria Teresa Lozano on the occasion of her 70th birthday}
\end{center}

\begin{abstract} $2$-stratifolds are a generalization of $2$-manifolds in that there are disjoint simple closed curves where several sheets meet. We show that the word problem for fundamental groups of $2$-stratifolds is solvable.\end{abstract}

\section{Introduction}  
    
Simple stratified spaces arise in Topological Data Analysis \cite{B},  \cite{CL}. A related class of $2$-complexes, called {\it $2$-foams}, has been defined and studied by Khovanov \cite{Ko} and Carter \cite{SC}. A special class of stratified spaces, called $2$-{\it stratifolds} have been introduced and some of their properties have been studied in \cite{GGH1}, \cite{GGH2}, \cite{GGH3} and similar spaces, called {\it multibranched surfaces}, have been investgated in \cite{MO}. A $2$-stratifold $X$ is a compact space with empty $0$-stratum and empty boundary and  contains a collection of finitely many disjoint simple closed curves, the components of the $1$-stratum $X^1$ of $X$, such that $X - X^1$ is a $2$-manifold and a neighborhood of each interval contained in $X^1$ consists of $n\geq 3$ sheets (the precise definition is given in section 2). A $2$-stratifold is essentially determined  by its associated labelled graph. In \cite{GGH1} it is shown that a simply connected $2$-stratifold is homotopy equivalent to a wedge of $2$-spheres and the simply connected $2$-stratifolds whose graph is a linear tree are classified. Furthermore an efficient algorithm (in terms of the associated graph) is developed for deciding whether a trivalent $2$-stratifold (where a neighborhood of each component $C$ of $X^1$ consists of $3$ sheets) is simply-connected and in \cite{GGH2} an efficient algorithm is given for deciding whether a given $2$-stratifold is homotopy equivalent to $S^2$.\\

Very few $2$-stratifolds occur as spines of closed $3$-manifolds. For example, fundamental groups of $3$-manifolds are residually finite, but there are simple $2$-stratifolds with non-residually finite fundamental group. Since $3$-manifold groups have solvable word problem (\cite{AFW}), the question arises whether this is true for $2$-stratifold groups. The main goal of this paper is to prove that this is indeed the case.


\section{Fundamental group of a graph of groups}

In this section we show that the word problem is solvable for fundamental groups of certain graphs of groups. A similar result for graphs of groupoids has been obtained by \cite{H}. Our proof for the graph of groups is more direct, using Serre's normal form. We first describe the fundamental group of a graph of groups $(G,\Gamma)$ following Serre \cite{Serre}. \\

A graph of groups $(G,\Gamma)$ consists of a graph $\Gamma$  with vertex set $vert\Gamma$ and (oriented) edge set $edge\Gamma$, an associated group $G_v$ to each $v\in vert\Gamma$ and a  group $G_e$ to each $e\in edge\Gamma$ such that $G_e =G_{\bar{e}}$, where  $\bar{e}$ is the inverse edge of $e$. (If $e\in \Gamma$, then $\bar{e}\in \Gamma$, $\bar{\bar{e}}=e$,  $e\neq \bar{e}$ and the initial edge $o(e)=t(\bar{e})$, the terminal edge of $\bar{e}$).  For each $e\in  edge\Gamma$  with terminal vertex $t(e)$ there is monomorphisms $\delta_{t(e)}:  G_e \to G_{t(e)}$. \\

The group $F(G,\Gamma )$ is generated by the groups $G_v$ ($v\in vert\Gamma$) and $edge\Gamma$, subject to the relations 

$\bar{e}=e^{-1}$ and  $e\delta_{t(e)} (a) e^{-1}=\delta_{o(e)} (a)$, for each edge $e\in edge\Gamma$ with initial edge $o(e)$ and terminal edge $t(e)$ and $a\in G_e$.\\

For a fixed vertex $v_0$, the fundamental group $\pi_1 (G,\Gamma ,v_0 )$ of the graph of groups $(G,\Gamma)$ is the subgroup of $F(G,\Gamma )$ generated by all words\\ 

$\omega =r_0 e_1 r_1 e_2 \dots e_n r_n$\\

where $v_0 \stackrel{e_1} - v_1 \stackrel{e_2} - v_2 - \dots \stackrel{e_n} - v_n$ is an edge path with initial and terminal vertex $v_0 =v_n$ (i.e. a cycle based at $v_0$) and $r_i \in G_{v_i}$.\\

The word $\omega =r_0 e_1 r_1 e_2 \dots e_n r_n$ of length $n$ is reduced, if

for $n=0$, $r_0 \neq 1 \in G_{v_0}$;

for $n\geq 2$, $r_i \not\in \delta_{t(e_{i})}(G_{e_i})$, for each index $i$ such that $e_{i+1}=\bar{e_i}$ (backtracking at vertex $v_i$). \\

Serre proves (\cite{Serre} Theorem 11):\\

If $\omega\in\pi_1 (G,\Gamma ,v_0 )$ is a reduced word then $\omega \neq 1$ in $\pi_1 (G,\Gamma ,v_0 )$.\\

\begin{thm} \label{graphgrps} Let $(G,\Gamma)$ be a graph of groups with finite graph $\Gamma$. Suppose that\\
(i) The word problem for each vertex group $G_v$ and each edge group $G_e$ is solvable.\\
(ii) For each edge $e$ of $G$, the membership problem with respect to $\delta_{t(e)} (G_e )$ is solvable in $G_{t(e)}$.\\
Then $\pi_1 (G,\Gamma ,v_0 )$ has a solvable word problem.

 \end{thm}

\begin{proof} Let $g \in \pi_1 (G,\Gamma ,v_0 )$ be represented by $\omega =r_0 e_1 r_1 e_2 \dots e_n r_n$, a word of length $n$.\\

If $n=0$ then $g=1$ if and only if $r_0 =1$ in $G_{v_0}$ and by (i) we can effectively decide whether this is the case.

If $n=1$ then $\omega$ is reduced and so $g\neq 1$.\\

If $n\geq 2$  we check if there is backtracking at $v_i$. If there is no backtracking at each $i=1,\dots ,n-1$, then $\omega$ is reduced and $g\neq 1$.

If there is backtracking at $v_i$ then by (ii) we can effectively check whether $r_i \in \delta_{t(e_{i})}(G_{e_i})$. If this is the case we find $a\in G_{e_i}$ such that $\delta_{t(e_{i})}(a)=r_i$. Then  $e_i r_i e_{i+1}=\delta_{o(e_{i})}(a)\in G_{v_{i-1}}$ and we replace $\omega =\dots r_{i-1}e_i r_i e_{i+1}r_{i+1} \dots $ by $\omega' =\dots (r_{i-1}\delta_{o(e_{i})}(a)r_{i+1}) \dots $, a word of length $n-2$ which represents the same $g\in \pi_1 (G,\Gamma ,v_0 )$. \\

Therefore we can effectively decide whether the word $\omega$ of length $n\geq 2$ representing $g$ is  reduced and, if $\omega$ is not  reduced, effectively find a word of length $n - 2$ representing the same $g$.

\end{proof}

\section{The graph of a $2$-stratifold.}

We first review the definition of a $2$-stratifold $X$ and its associated graph $G_X$ given in \cite{GGH1}. A  $2$-{\it stratifold} is a compact, Hausdorff space $X$ that contains a closed (possibly disconnected) $1$-manifold $X^1$ as a closed subspace with the following property: Each  point $x\in X^1$  has a neighborhood homeomorphic to $\mathbb{R}{\times}CL$, where $CL$ is the open cone on $L$ for some (finite) set $L$ of cardinality $>2$  and $X - X^1$ is a (possibly disconnected) open $2$-manifold.\\

A component $C\approx S^1$ of $X^1$ has a regular neighborhood $N(C)= N_{\pi}(C)$ that is homeomorphic to $(Y {\times}[0,1]) /(y,1)\sim (h(y),0)$, where $Y$ is the closed cone on the discrete space $\{1,2,...,d\}$ and $h:Y\to Y$ is a homeomorphism whose restriction to $\{1,2,...,d\}$ is the permutation $\pi:\{1,2,...,d\}\to  \{1,2,...,d\}$. The space $N_{\pi}(C)$ depends only on the conjugacy class of $\pi \in S_d$ and therefore is determined by a partition of $d$. A component of $\partial N_{\pi}(C)$ corresponds then to a summand of the partition determined by $\pi$. Here the neighborhoods $N(C)$ are chosen sufficiently small so that for disjoint components $C$ and $C'$ of $X_1$, $N(C)$ is disjoint from $N(C' )$. The components of $\overline{N(C)-C}$ are called the {\it sheets} of $N(C)$.\\

The associated labelled graph $G=G_X$ of a given $2$-stratifold $X=X_G$ is a bipartite graph with black vertices and labelled white vertices and edges. The white vertices $w$ of $G_X$ are the components $W$ of $M:=\overline{X-\cup_j N(C_j)}$ where $C_j$  runs over the components of $X^1$;  the black vertices $b_j$ are the $C_j$'s. An edge $e$ corresponds to a component $S$ of $\partial M$; it joins a white vertex $w$ corresponding to $W$ with a black vertex $b$ corresponding to $C_j $ if $S=W\cap N (C_j)$. Note that the number of boundary components of $W$ is the number of adjacent edges of $W$. \\

The label assigned to a white vertex $W$ is its genus $g$;  the label of an edge $e$ is an integer $m$, where $|m|$ is the summand of the partition $\pi$ corresponding to the component $S\subset \partial N_{\pi}(C)$ and the sign of $m$ is determined by an orientation of $C_j$ and $S$. (Here we use Neumann's \cite{N} convention of assinging negative genus $g$ to nonorientable surfaces; for example the genus $g$ of the projective plane or the Moebius band is $-1$, the genus of the Klein bottle is $-2$).  Note that the partition $\pi$ of a black vertex is determined by the labels of the adjacent edges. If all white vertices have labels $g<0$ or if $G$ is a tree, then the labeled graph determines $X$ uniquely. \\

\section{Natural presentation of $\pi_1 (X_G )$}

In this section we describe a natural presentation for the fundamental group of a $2$-stratifold $X$. First we fix a notation.\\

\noindent $X=X_G$ is a $2$- stratifold  with 
associated bipartite graph $G=G_X$.\\

\noindent $N(C_{b_j})$ is a regular neighborhood of $C_{b_j}$, a component of $X^{(1)}$ corresponding to the black vertex $b_j$ of $G_X$\\

\noindent $W_i$ is a component of $M=\overline{X-\cup_j N(C_j)}$ corresponding to the white vertex $w_i$ of $G_X$\\

\noindent $c_{ijk}$ are the components of $W_i \cap N(C_j )$ corresponding to the edges $e_{ijk}$ of $G_X$\\


\noindent For a given white vertex $w$, the compact $2$-manifold $W$ has conveniently oriented boundary curves $c_1 ,\dots , c_p$ such that\\

 (*) \qquad $\pi_1 (W)=\langle c_1 ,\dots, c_p , y_1 ,\dots , y_n : c_1 \cdots c_p \cdot q =1 \rangle $\\
 
 \noindent  where $q=[y_1 ,y_2 ]\dots [y_{2g-1},y_{2g}]$,  if $W$ is orientable of genus $g$ and $n=2g$,
 
\noindent $q=y_1^2 \dots y_n^2$,  if $W$ is non-orientable of genus $-n$.\\

\noindent Let $\mathcal{B}$ be the set of black vertices, $\mathcal{W}$ the set of white vertices and choose a fixed maximal tree $T$ of $G$.\\ 

\noindent We choose  orientations of the black vertices and of all boundary components of $M$ such that all labels of edges in $T$ are positive.\\

Then $\pi_1 (X_G )$ has a natural presentation with  \\

 \noindent \begin{tabular}{rl}
 generators: & $\{b\}_{b\in \mathcal{B}}$    \\
 
    & $\{c_1 ,\dots, c_p , y_1 ,\dots , y_n\}$, one set for each $w\in \mathcal{W}$,  as in $(*)$\\
   
   &  $\{t_i\}$, one $t_i$ for each  edge $c_i \in G-T$ between $w$ and $b$ \\
   
\end{tabular} \\

\noindent \begin{tabular}{rl}
and relations: & $c_1 \cdots c_p \cdot q =1$, one for each $w\in \mathcal{W}$,  as in $(*)$\\

&  $b^m =c_i$, for each edge $c_i \in T$ between $w$ and $b$ with label $m\geq 1$ \\
& (corresponding to $W \cap N(C_b )$)\\

& $t_i^{-1}c_i t_i=b^{m_i}$, for each edge $c_i \in G-T$ between $w$ and $b$ with label $m_i \in \mathbb{Z}$. \\
\end{tabular} \\

\section{The graph of groups of $X_G$}

Let $X=X_G$ be the 2-stratifold associated to the labeled  graph G; we assume a maximal tree $T$ of $G$ is given and the labels of edges of $T$ are positive, so the labeling is unique. 
We first define a graph of CW-complexes as in \cite{SW}, with underlying graph that of $G$. \\

For a black vertex $b$ representing a singular oriented circle $C_b$, let $o(b)$ be the order of $C_b$ in $\pi_1 (X_G)$. Note that, if $e$ is an edge joining a black vertex $b$ to a white vertex $w$  and the label of $e$ is $m$, then $e$ represents an oriented circle $c$ of $\partial W$ whose order in $\pi_1 (X_G)$ is  $k=o(b)/(o(b),m)$. Here $(o(b),m)$ denotes the greatest common divisor of $o(b)$ and $m$. (If $o(b)=0$, then $(o(b),m)=m$).\\

Construct a space ${\hat X}$ from $X$ by attaching disks as follows:

If $b$ is a black vertex of order $o(b)\geq 1$, attach a 2-cell $d_b$ to $C_b$ with degree $o(b)$ (i.e. attach a disk under the attaching map $z\to z^{o(b)}$). If $e$ is an edge joining $b$ to $w$ with label $m$ and $o(b)\geq 1$, attach to $c$ a 2-cell $d_e$ with degree $k=o(b)/(o(b),m)$. If $o(b)=0$, do not attach  2-cells). Note that \\

$\pi(\hat{X}) = \pi(X_G)$.\\

The graph of spaces associated to $\hat{X}$ has the same underlying graph as $G_X$, with vertices $\hat{X}_b$, $\hat{X}_w$, and edges $\hat{X}_e$, defined as follows:\\
 
$\hat{X}_b$: For a black vertex $b$ of $G$, $\hat{X}_b = N(C_b) \cup d_b \cup ( \cup d_e )$, where $e$ runs over the edges having $b$ as an endpoint. 

$\hat{X}_e$: For a white vertex $w$ of $G$ let $\hat{X}_w$,  $W\cup (\cup d_e )$, where $e$ runs over the edges incident to $w$. (Recall that there is one such edge for each boundary curve $c$ of $W$).

$\hat{X}_e$: For an edge $e$  joining $b$ to $w$, $\hat{X}_e = (\hat{X}_b \cap \hat{X}_w)$. (If $o(b)\geq 1$ this is a pseudo-projective plane of degree $k=o(b)/(o(b),m)$).\\

Since $\hat{X}_b$, $ \hat{X}_w$ and $\hat{X}_e$ are path-connected and the  inclusion-induced homomorphisms $\pi_1 (\hat{X}_e) \to \pi_1 (X_b)$ and $\pi_1 (\hat{X}_e) \to \pi_1 (\hat{X}_w)$
are injective, this graph of spaces determines a graph of groups $\mathcal{G}=\{G_b ,G_e ,G_w \}$ 
(with the same underlying graph as $G_X$). The vertex groups are $G_b = \pi_1 (\hat{X}_b)$ and $G_w = \pi_1 (\hat{X}_w)$, the edge groups are $G_e = \pi_1 (\hat{X}_e)$, the monomorphisms $\delta: G_e \to G_b$ (resp. $G_e \to G_w$ are induced by inclusion. Then (see for example \cite{SW},\cite{Serre})\\

 $\pi_1 \mathcal{G}\cong \pi_1 (\hat{X})$ \\

Note that the groups $G_b$ of the black vertices and the groups $G_e$ of the edges are cyclic. For a white vertex $w$ with edges $e_1 ,\dots e_p$ labelled $m_1 ,\dots m_p$ with associated edge space $X_w =W\cup_{i=1}^r d_{e_i}$ we have \\

$G_w =\pi_1 (\hat{X}_w )=\langle c_1 ,\dots, c_p , y_1 ,\dots , y_n : c_1 \cdots c_p \cdot q =1\,,\,c_1^{k_1}=\dots =c_1^{k_r}= 1\,\,(1\leq r\leq p)  \, \rangle $.\\

If all $k_i \geq 2$ and $r=p$ then $G_w$ is an $F$-group (\cite{LS} p. 126-127), otherwise it is a free product of cyclic groups.


\section{The Word Problem for Fundamental groups of $2$-stratifolds}

It is well-known that free groups have solvable membership problem with respect to cyclic subgroups. More generally it follows from the Proposition below, which is Corollary 4.16 of \cite{AFW}, that free products of cyclic groups have solvable membership problems. 

\begin{prop} \label{freeproducts} Solvability of the membership problem is preserved under taking free products.
\end{prop}

 We are interested in the membership problem of free products with amalgamation with respect to cyclic groups and give an elementary proof of the following

\begin{lem} \label{cyclicsubgrps} Let $G=A*_C B$. Assume that $C$ has solvable membership problem with respect to cyclic subgroups and $A$ and $B$ have solvable membership problem with respect to the subgroup $C$. Then $G$ has solvable membership problem with respect to cyclic subgroups.
\end{lem}

\begin{proof} For $g, g' \in A$ or $B$ we can decide whether $g(g')^{-1}$ is in  $C$. Therefore, given $g\in G$ and a fixed choice of right coset representatives of $C$ in $A$ (resp. in $B$) we can effectively find the (unique) reduced normal form $w=g_1 \dots g_n c$ of $g$, where $g_i \in A$ or $B$ are the  chosen representatives of the right cosets $g_i C$, $c\in C$,  and $g_i $, $g_{i+1}$ are in different subgroups $A$, $B$, for $i=1,\dots , n-1$. The length of $g$ is $l(g)=l(w)=n$. In particular, $l(w)=0$ iff $g\in C$. Also, if $w$ is not cyclically reduced (i.e. $g_1$ and $g_n$ are in the same subgroup $A$ or $B$), then we can effectively reduce $w$ to a cyclically reduced word.

Let $t\in G$ of length $l(t)\geq 0$ generate an infinite cyclic subgroup  $\langle t \rangle\subset G$ and let $g\in G$. Now $w\in \langle t \rangle$ if and only if $w=t^k$ for some $|k|\geq 1$. Since $w\in \langle t \rangle$ iff $w^{-1}\in\langle t \rangle$ we may assume $k\geq 1$. If $l(t)=0$ then $l(w)=0$ and the result follows since $C$ has solvable membership problem with respect to cyclic groups. Thus assume $l(t)\geq1$.

If the word $t$ is cyclically reduced then $l(t^k )=kl(t)$. Thus  there is a unique $k$ such that $l(w)=kl(t)$ and we can effectively check whether the reduced words $w$ and $t^k$ agree.

If $t$ is not cyclically reduced then $t=uru^{-1}$ for some reduced word $r$ and cyclically reduced word $r$. Then $w\in \langle t \rangle$ iff $u^{-1}wu=r^k$ for some $k$. We effectively find the reduced word $w'$ representing $u^{-1}wu$ and (by the above argument) effectively determine whether $w'=r^k$.
\end{proof}
\begin{cor} Let $G$ be a free product of cyclic groups or a free product of two such groups amalgamated over a cyclic group. Then the membership problem with respect to cyclic subgroups is solvable.
\end{cor}
\begin{proof} This follows from Proposition \ref{freeproducts} and Lemma \ref{cyclicsubgrps}.
\end{proof}
\begin{cor} \label{Fgrps}$F$-groups have solvable membership problem with respect to cyclic subgroups.
\end{cor}
\begin{proof} Let $G=\langle c_1 ,\dots, c_p , y_1 ,\dots , y_n : c_1 \cdots c_p \cdot q =1\,,\,c_1^{k_1}=\dots =c_1^{k_p}= 1\,\rangle $ be an $F$-group. 

If $p\geq 1$, let $A=\langle c_1 ,\dots, c_p :  c_1^{k_1}= \dots =c_1^{k_p}= 1\rangle$, $B= \langle y_1 ,\dots , y_n :   \, \rangle $,  $C$ the infinite cyclic group generated by $c_1 \cdots c_p$ in $A$ resp. by $q$ in $B$. Then $G=A*_C B$.

If $p=0$ then $G$ is the fundamental group of a closed surface of genus $g$. If $g\neq 1, -2$, then $G$ can be similarly written as a free product of two free groups with amalgamation over a cyclic group. If $g= 1, -2$ the result is trivial (in this case every element of $G$ has a normal form of length $\leq 2$).
\end{proof}
  
\begin{thm}  The fundamental group  of a $2$-stratifold has solvable word problem.
\end{thm}

\begin{proof} From section 5 we know that $\pi(X_G)\cong \pi_1 \mathcal{G}$ for a graph of groups where the edge groups and black vertex groups are cyclic and the white vertex groups are $F$-groups or free products of finitely many cyclic groups. By Lemma \ref{cyclicsubgrps} and Corollary \ref{Fgrps} all these groups have solvable membership problem with respect to finite cyclic subgroups. Now the Theorem follows from Theorem \ref{graphgrps}.

%
\end{proof}

\section{Some Consequences}

\begin{cor} \label{s.c.} There is an algorithm to decide whether or not $\pi(X_G )$ is abelian.
\end{cor}

\begin{proof} $\pi(X_G )$ is abelian if and only if $[x_i ,x_j ]=1$ for $1\leq i<j\leq n]$, where $x_1 ,\dots ,x_n$ generate $\pi(X_G )$. Since the word problem is solvable, we can decide whether this is true. 
\end{proof}

A $0$-terminal edge of $X_G$ is an edge $b \stackrel{m}- w$, where $w$ is a terminal white vertex of genus $0$. The following deals with a special case of the order problem.

\begin{cor} \label{order} Let $b \stackrel{m}- w$ be a $0$-terminal edge of $G_X$. One can calculate the (finite) order $o$ of $b$ in $\pi(X_G)$.
\end{cor}

\begin{proof} $o$ is one of the (finitely many) divisors of the finite nonzero labels of $b - w$. The Corollary follows since $\pi(X_G)$ has solvable word problem.
\end{proof}

In \cite{GGH1} and \cite{GGH3} we obtained for certain classes of $2$-stratifolds $X_G$ (namely those with a linear graph $G_X$ or those that are trivalent) an {\it efficient} algorithm to decide if $X_G$ is simply connected. These algorithms can be read off from the labelled graph $G_X$. For the general case we do not yet have an {\it efficient} algorithm, but we now see that there is an algorithm:
 
\begin{cor} \label{s.c.} There is an algorithm to decide whether or not $X_G$ is simply-connected.
\end{cor}

\begin{proof} If $S$ is a finite set of generators of $\pi=\pi_1 (X_G )$, the $\pi =1$ if and only if  $s=1$ in $\pi$ for every $s\in S$. Since $\pi$ has solvable word problem one can decide if every $s$ in $S$ is $1$.
\end{proof}

In \cite{GGH1} it was shown that a necessary condition for a $2$-stratifold $X_G$ to be simply-connected is that $G_X$ is a tree, all white vertices are of genus $0$, and all terminal edges are white. If there is an efficient algorithm for the order problem in Corollary \ref{order} then this result may be used in obtaining an efficient algorithm in Corollary \ref{s.c.} as follows:\\

If $G_X$ is not a tree or if some white vertex has nonzero genus, or if there is a black terminal vertex, then $\pi(X_G)\neq 1$.  Otherwise apply repeatedly the following ``pruning":\\

Calculate the order $o$ in $\pi(X_G)$ of $b$ where $b - w$ is a $0$-terminal edge ; if $o$ is not $1$ then $X_G$ is not simply-connected; if $o =1$ delete $b$, $w$ and all edges incident to $b$ from $X_G$. Each component $G_i$  of the resulting graph (the ``pruned" graph) corresponds to a $2$-stratifold $X_{G_i}$, and since $o(b)=1$ in $\pi(X_G)$ it follows that $X_G$ is simply-connected if and only if each $X_{G_i}$ is simply-connected. Then  $X_G$ is 1-connected if and only if we eventually obtain a graph with no edges.\\

\begin{cor} One can decide whether or not $X_G$ is homotopically equivalent to a wedge of $n$ $2$-spheres and, if so, calculate $n$.
\end{cor}

\begin{proof} In \cite{GGH2} it was shown that a simply-connected $2$-stratifold $X_G$ is homotopy equivalent to a wedge of $2$-spheres and moreover if $n_b$ (resp. $n_w$) denotes the number of black (resp. white) vertices of $G_X$, then $X_G$  is homotopy equivalent to a wedge of $n_w - n_b$ $2$-spheres. Now the Corollary  follows from Corollary \ref{s.c.} and, if $\pi(X_G) =1$, then $n = n_w - n_b$. 
\end{proof}

{\it Acknowledgments:} J. C. G\'{o}mez-Larra\~{n}aga would like to thank LAISLA and the TDA project from CIMAT for  financial support and IST Austria for their hospitality.


\begin{thebibliography}{99}

\bibitem {AFW} M. Aschenbrenner, S. Friedel, H. Wilson, Decision Problems for $3$-manifolds and their fundamental groups, arXiv:1405.6274v2 [math.GT], (2015).

\bibitem {B} P. Bendich, E. Gasparovicy, C.J. Traliez, J. Harer, Scaffoldings and Spines: Organizing High-Dimensional Data Using Cover Trees, Local Principal Component Analysis, and Persistent Homology, arXiv:1602.06245v2 [cs.CG] 27 Feb 2016.

\bibitem {SC} J.S. Carter, Reidemeister/Roseman-type moves to embedded foams in 4-dimensional space. arXiv:1210.3608v1 [math.GT] 

\bibitem {GGH1} J.C. G\'{o}mez-Larra\~{n}aga, F. Gonz\'alez-Acu\~na, Wolfgang Heil, $2$-stratifolds, in ``A Mathematical Tribute to Jos\'e Mar\'ia Montesinos Amilibia", Universidad Complutense de Madrid, 395-405 (2016).

\bibitem {GGH2} J.C. G\'{o}mez-Larra\~{n}aga, F. Gonz\'alez-Acu\~na, Wolfgang Heil, $2$-dimensional stratifolds homotopy equivalent to $S^2$, Topology Appl. 209, 56-62 (2016).

\bibitem {GGH3} J.C. G\'{o}mez-Larra\~{n}aga, F. Gonz\'alez-Acu\~na, Wolfgang Heil, Classification and models of simply-connected trivalent $2$-dimensional stratifolds,  arXiv:1611.08013 (2016)

\bibitem{H} P.K.J. Horadam, The Word problem and related results for graph products of groups, Proc. Amer. Math. Soc. 82 (1981), 157-164.

\bibitem {Ko} M. Khovanov,  sl(3) link homology. Algebr. and Geom. Topol. 4, 1045-1081 (2004).

\bibitem {CL} P. Lum, G. Singh, J. Carlsson, A. Lehman, T. Ishkhanov, M. Vejdemo-Johansson, M. Alagappan, G. Carlsson, Extracting insights from the shape of complex data using topology. Nature Scientific Reports 3, 12-36 (2013).
\bibitem {LS} R. C. Lyndon and P. E. Schupp, Combinatorial Group Theory, Modern Surveys in Math., no. 89, Springer Verlag, Berlin, 1977.

\bibitem{MO} S. Matsuzaki and M. Ozawa, Genera and minors of multibranched surfaces, arXiv:1603.09041v1 [math.GT] 30 Mar 2016.

\bibitem {N}  W. Neumann, A calculus for plumbing applied to the topology of complex surface singularities and degenerating complex curves, Trans. Amer. Math. Soc. 268, 299-344 (1981).

\bibitem {SW} P. Scott and C.T.C.Wall, Topological Methods in Group Theory, In Homological Group Theory, London Math. Soc. Lecture Notes Ser. 36, Cambridge Univ. Press (1979).

\bibitem{Serre} J.P. Serre, Trees, Springer-Verlag, 1980. 









\end{thebibliography}
\end{document}